\documentclass[twoside,american,DIV12]{scrartcl}
\usepackage[T1]{fontenc}
\usepackage[latin1]{inputenc}
\usepackage{babel}
\usepackage{prettyref}
\usepackage{mathtools}
\usepackage{amsthm}
\usepackage{amsmath}
\usepackage{amssymb}
\usepackage[unicode=true,pdfusetitle,
 bookmarks=true,bookmarksnumbered=false,bookmarksopen=false,
 breaklinks=false,pdfborder={0 0 1},backref=false,colorlinks=false]
 {hyperref}

\makeatletter
 \theoremstyle{definition}
 \newtheorem*{defn*}{\protect\definitionname}
\theoremstyle{plain}
\newtheorem{thm}{\protect\theoremname}[section]
  \theoremstyle{remark}
  \newtheorem{rem}[thm]{\protect\remarkname}
  \theoremstyle{definition}
  \newtheorem{example}[thm]{\protect\examplename}
  \theoremstyle{plain}
  \newtheorem{lem}[thm]{\protect\lemmaname}
  \theoremstyle{plain}
  \newtheorem{prop}[thm]{\protect\propositionname}

\usepackage[T1]{fontenc}    
\usepackage{enumerate}
\usepackage{a4wide}        
\usepackage{amsmath}
\usepackage{euscript}
\usepackage{url}
\usepackage{prettyref}
\usepackage{scrpage2}
\usepackage{tu-preprint}
\usepackage{scrpage2}
\allowdisplaybreaks[1] 
\usepackage{amsfonts}
\usepackage{amsfonts}
\newenvironment{keywords}{ \noindent\footnotesize\textbf{Keywords and phrases:}}{}

\newenvironment{class}{\noindent\footnotesize\textbf{Mathematics subject classification 2010:}}{}

\usepackage{color}

\newcommand*{\dive}{\operatorname{div}}

\newcommand*{\curl}{\operatorname{curl}}

\newcommand*{\grad}{\operatorname{grad}}

\DeclareMathAccent{\Circ}{\mathalpha}{operators}{"17}

\renewcommand{\Re}{\operatorname{\mathfrak{Re}}}

\renewcommand*{\epsilon}{\varepsilon}
\renewcommand*{\rho}{\varrho}

\newrefformat{prop}{Proposition \ref{#1}}
\newrefformat{lem}{Lemma \ref{#1}}
\newrefformat{thm}{Theorem \ref{#1}}
\newrefformat{cor}{Corollary \ref{#1}}
\newrefformat{rem}{Remark \ref{#1}}
\newrefformat{exa}{Example \ref{#1}}
\newrefformat{sub}{Subsection \ref{#1}}
\newrefformat{eq}{(\ref{#1})}

\newrefformat{item}{(\ref{#1})}

\theoremstyle{definition}
\newtheorem*{hyp}{Hypothesis}

\newrefformat{hyp}{Hypothesis \ref{#1}}

\makeatother

\institut{Institut f\"ur Analysis}

\preprintnumber{MATH-AN-01-2015}

\preprinttitle{On a class of block operator matrices in system theory.}

\author{Sascha Trostorff}

\AtBeginDocument{
  
}

\makeatother

  \providecommand{\definitionname}{Definition}
  \providecommand{\examplename}{Example}
  \providecommand{\lemmaname}{Lemma}
  \providecommand{\propositionname}{Proposition}
  \providecommand{\remarkname}{Remark}
\providecommand{\theoremname}{Theorem}

\begin{document}
\makepreprinttitlepage

\author{ Sascha Trostorff \\ Institut f\"ur Analysis, Fachrichtung Mathematik\\ Technische Universit\"at Dresden\\ Germany\\ sascha.trostorff@tu-dresden.de}

\title{On a class of block operator matrices in system theory.}

\maketitle
\begin{abstract} \textbf{Abstract. }We consider a class of block
operator matrices arising in the study of scattering passive systems,
especially in the context of boundary control problems. We prove that
these block operator matrices are indeed a subclass of block operator
matrices considered in {[}Trostorff: A characterization of boundary
conditions yielding maximal monotone operators. J. Funct. Anal., 267(8):
2787--2822, 2014{]}, which can be characterized in terms of an associated
boundary relation.\end{abstract}


\begin{class} 47N20, 47B44, 93A30 \end{class}

\newpage

\tableofcontents{} 

\newpage

\section{Introduction}

Following \cite{Staffans2012_phys}, a natural class of $C_{0}$-semigroup
generators $-A$ arising in the context of scattering passive systems
in system theory, can be described as a block operator matrix of the
following form: Let $E_{0},E,H,U$ be Hilbert spaces with $E_{0}\subseteq E$
dense and continuous and let $L\in L(E_{0},H),\, K\in L(E_{0},U).$
Moreover, denote by $L^{\diamond}\in L(H,E_{0}')$ and $K^{\diamond}\in L(U,E_{0}')$
the dual operators of $L$ and $K,$ respectively, where we identify
$H$ and $U$ with their dual spaces. Then $A$ is a restriction of
$\left(\begin{array}{cc}
K^{\diamond}K & -L^{\diamond}\\
L & 0
\end{array}\right)$ with domain 
\begin{equation}
\mathcal{D}(A)\coloneqq\left\{ (u,w)\in E_{0}\times H\,|\, K^{\diamond}Ku-L^{\diamond}w\in E\right\} ,\label{eq:domain_syst}
\end{equation}
where we consider $E\cong E'$ as a subspace of $E_{0}'$. It is proved
in \cite[Theorem 1.4]{Staffans2012_phys} that for such operator matrices
$A$, the operator $-A$ generates a contractive $C_{0}$-semigroup
on $E\oplus H$ and a so-called scattering passive system, containing
$-A$ as the generator of the corresponding system node, is considered
(see \cite{Staffans2005_book} for the notion of system nodes and
scattering passive systems). This class of semigroup generators were
particularly used to study boundary control systems, see e.g. \cite{Staffans2013_Maxwell,Tucsnak_Weiss2003,Tucsnak2003_thinair2}.
In these cases, $L$ is a suitable realization of a differential operator
and $K$ is a trace operator associated with $L$. More precisely,
$G_{0}\subseteq L\subseteq G,$ where $G_{0}$ and $G$ are both densely
defined closed linear operators, such that $K|_{\mathcal{D}(G_{0})}=0$
(as a typical example take $G_{0}$ and $G$ as the realizations of
the gradient on $L_{2}(\Omega)$ for some open set $\Omega\subseteq\mathbb{R}^{n}$
with $\mathcal{D}(G_{0})=H_{0}^{1}(\Omega)$ and $\mathcal{D}(G)=H^{1}(\Omega)$).
It turns out that in this situation, the operator $A$ is a restriction
of the operator matrix $\left(\begin{array}{cc}
0 & D\\
G & 0
\end{array}\right),$ where $D\coloneqq-(G_{0})^{\ast}$ (see \prettyref{lem:D_extends}
below). Such restrictions were considered by the author in \cite{Trostorff2013_bd_maxmon},
where it was shown that such (also nonlinear) restrictions are maximal
monotone (a hence, $-A$ generates a possibly nonlinear contraction
semigroup), if and only if an associated boundary relation on the
so-called boundary data space of $G$ is maximal monotone.\\
In this note, we characterize the class of boundary relations, such
that the corresponding operator $A$ satisfies \prettyref{eq:domain_syst}
for some Hilbert spaces $E_{0},U$ and operators $L\in L(E_{0},H),\, K\in L(E_{0},U)$.
We hope that this result yields a better understanding of the semigroup
generators used in boundary control systems and provides a possible
way to generalize known system-theoretical results to a class of nonlinear
problems.\\
The article is structured as follows. In Section 2 we recall the basic
notion of maximal monotone relations, we state the characterization
result of \cite{Trostorff2013_bd_maxmon} and introduce the class
of block operator matrices considered in \cite{Staffans2012_phys}.
Section 3 is devoted to the main result (\prettyref{thm:main}) and
its proof.\\
Throughout, every Hilbert space is assumed to be complex, its inner
product $\langle\cdot|\cdot\rangle$ is linear in the second and conjugate
linear in the first argument and the induced norm is denoted by $|\cdot|$.

\section{Preliminaries}

\subsection{Maximal monotone relations}

In this section we introduce the basic notions for maximal monotone
relations. Throughout let $H$ be a Hilbert space. 
\begin{defn*}
Let $C\subseteq H\oplus H.$ We call $C$ \emph{linear, }if $C$ is
a linear subspace of $H\oplus H$. Moreover, we define for $M,N\subseteq H$
the \emph{pre-set of $M$ under $C$ }by\emph{ 
\[
[M]C\coloneqq\{x\in H\,|\,\exists y\in M:(x,y)\in C\}
\]
}and the \emph{post-set of $N$ under $C$ }by 
\[
C[N]\coloneqq\{y\in H\,|\,\exists x\in N:(x,y)\in C\}.
\]
The \emph{inverse relation $C^{-1}$ }of $C$ is defined by 
\[
C^{-1}\coloneqq\{(v,u)\in H\oplus H\,|\,(u,v)\in C\}.
\]
A relation $C$ is called \emph{monotone, }if for each $(x,y),(u,v)\in C$:
\[
\Re\langle x-u|y-v\rangle\geq0.
\]
A monotone relation $C$ is called \emph{maximal monotone, }if for
each monotone relation $B\subseteq H\oplus H$ with $C\subseteq B$
we have $C=B.$ Moreover, we define the \emph{adjoint relation $C^{\ast}\subseteq H\oplus H$
of $C$ }by 
\[
C^{\ast}\coloneqq\left\{ (v,-u)\in H\oplus H\,|\,(u,v)\in C\right\} ^{\bot},
\]
where the orthogonal complement is taken in $H\oplus H$. A relation
$C$ is called \emph{selfadjoint, }if $C=C^{\ast}$.\end{defn*}
\begin{rem}
$\,$

\begin{enumerate}[(a)]

\item A pair $(x,y)\in H\oplus H$ belongs to $C^{\ast}$ if and
only if for each $(u,v)\in C$ we have 
\[
\langle v|x\rangle_{H}=\langle u|y\rangle_{H}.
\]
Thus, the definition of $C^{\ast}$ coincides with the usual definition
of the adjoint operator for a densely defined linear operator $C:\mathcal{D}(C)\subseteq H\to H.$ 

\item Note that a selfadjoint relation is linear and closed, since
it is an orthogonal complement.

\end{enumerate}
\end{rem}
We recall the famous characterization result for maximal monotone
relations due to G. Minty.
\begin{thm}[\cite{Minty}]
 \label{thm:minty} Let $C\subseteq H\oplus H$ be monotone. Then
the following statements are equivalent:

\begin{enumerate}[(i)]

\item $C$ is maximal monotone,

\item $\exists\lambda>0:\,(1+\lambda C)[H]=H,$ where $(1+\lambda C)\coloneqq\{(u,u+\lambda v)\,|\,(u,v)\in C\}$,

\item $\forall\lambda>0:\,(1+\lambda C)[H]=H.$

\end{enumerate}\end{thm}
\begin{rem}
$\,$\label{rem:max_mon}

\begin{enumerate}[(a)]

\item We note that for a monotone relation $C\subseteq H\oplus H$
the relations $(1+\lambda C)^{-1}$ for $\lambda>0$ are Lipschitz-continuous
mappings with best Lipschitz-constant less than or equal to $1$.
By the latter Theorem, maximal monotone relations are precisely those
monotone relations, where $(1+\lambda C)^{-1}$ for $\lambda>0$ is
defined on the whole Hilbert space $H$. 

\item If $C\subseteq H\oplus H$ is closed and linear, then $C$
is maximal monotone if and only if $C$ and $C^{\ast}$ are monotone.
Indeed, if $C$ is maximal monotone, then $(1+\lambda C)^{-1}\in L(H)$
for each $\lambda>0$ with $\sup_{\lambda>0}\|(1+\lambda C)^{-1}\|\leq1.$
Hence, $(1+\lambda C^{\ast})^{-1}=\left(\left(1+\lambda C\right)^{-1}\right)^{\ast}\in L(H)$
for each $\lambda>0$ with $\sup_{\lambda>0}\|(1+\lambda C^{\ast})^{-1}\|\leq1$.
The latter gives for each $(x,y)\in C^{\ast}$ and $\lambda>0$
\begin{align*}
|x+\lambda y|_{H}^{2} & =|x|_{H}^{2}+2\Re\lambda\langle x|y\rangle_{H}+\lambda^{2}|y|_{H}^{2}\\
 & =|(1+\lambda C^{\ast})^{-1}(x+\lambda y)|_{H}^{2}+2\Re\lambda\langle x|y\rangle_{H}+\lambda^{2}|y|_{H}^{2}\\
 & \leq|x+\lambda y|_{H}^{2}+2\Re\lambda\langle x|y\rangle_{H}+\lambda^{2}|y|_{H}^{2}
\end{align*}
and hence, 
\[
-\frac{\lambda}{2}|y|^{2}\leq\Re\langle x|y\rangle_{H}.
\]
Letting $\lambda$ tend to $0$, we obtain the monotonicity of $C^{\ast}$.
If on the other hand $C$ and $C^{\ast}$ are monotone, we have that
$[\{0\}](1+\lambda C^{\ast})=\{0\}$ for each $\lambda>0$ and thus,
$\overline{[H](1+\lambda C)^{-1}}=\overline{(1+\lambda C)[H]}=\left([\{0\}](1+\lambda C^{\ast})\right)^{\bot}=H$.
Since, moreover $(1+\lambda C)^{-1}$ is closed and Lipschitz-continuous
due to the monotonicity of $C$, we obtain that $[H](1+\lambda C)^{-1}$
is closed, from which we derive the maximal monotonicity by \prettyref{thm:minty}.

\end{enumerate}
\end{rem}

\subsection{Boundary data spaces and a class of maximal monotone block operator
matrices}

In this section we will recall the main result of \cite{Trostorff2013_bd_maxmon}.
For doing so, we need the following definitions. Throughout, let $E,H$
be Hilbert spaces and $G_{0}:\mathcal{D}(G_{0})\subseteq E\to H$
and $D_{0}:\mathcal{D}(D_{0})\subseteq H\to E$ be two densely defined
closed linear operators satisfying 
\[
G_{0}\subseteq-\left(D_{0}\right)^{\ast}.
\]
We set $G\coloneqq\left(-D_{0}\right)^{\ast}\supseteq G_{0}$ and
$D\coloneqq-\left(G_{0}\right)^{\ast}\supseteq D_{0}$, which are
both densely defined closed linear operators.
\begin{example}
\label{exa:G_D}As a guiding example we consider the following operators.
Let $\Omega\subseteq\mathbb{R}^{n}$ open and define $G_{0}$ as the
closure of the operator 
\begin{align*}
C_{c}^{\infty}(\Omega)\subseteq L_{2}(\Omega) & \to L_{2}(\Omega)^{n}\\
\phi & \mapsto\left(\partial_{i}\phi\right)_{i\in\{1,\ldots,n\},}
\end{align*}
where $C_{c}^{\infty}(\Omega)$ denotes the set of infinitely differentiable
functions compactly supported in $\Omega$. Moreover, let $D_{0}$
be the closure of 
\begin{align*}
C_{c}^{\infty}(\Omega)\subseteq L_{2}(\Omega)^{n} & \to L_{2}(\Omega)\\
(\phi_{i})_{i\in\{1,\ldots,n\}} & \mapsto\sum_{i=1}^{n}\partial_{i}\phi_{i}.
\end{align*}
Then, by integration by parts, we obtain $G_{0}\subseteq-\left(D_{0}\right)^{\ast}$.
Moreover, we have that $G:\mathcal{D}(G)\subseteq L_{2}(\Omega)\to L_{2}(\Omega)^{n},\, u\mapsto\grad u$
with $\mathcal{D}(G)=H^{1}(\Omega)$ as well as $D:\mathcal{D}(D)\subseteq L_{2}(\Omega)^{n}\to L_{2}(\Omega),\, v\mapsto\dive v$
with $\mathcal{D}(D)=\{v\in L_{2}(\Omega)^{n}\,|\,\dive v\in L_{2}(\Omega)\},$
where $\grad u$ and $\dive v$ are meant in the sense of distributions.
We remark that in case of a smooth boundary $\partial\Omega$ of $\Omega,$
elements $u\in\mathcal{D}(G_{0})=H_{0}^{1}(\Omega)$ are satisfying
$u=0$ on $\partial\Omega$ and elements $v\in\mathcal{D}(D_{0})$
are satisfying $v\cdot n=0$ on $\partial\Omega$, where $n$ denotes
the unit outward normal vector field. Thus, $G_{0}$ and $D_{0}$
are the gradient and divergence with vanishing boundary conditions,
while $G$ and $D$ are the gradient and divergence without any boundary
condition. In the same way one might treat the case of $G_{0}=\curl_{0}$,
the rotation of vectorfields with vanishing tangential component and
$G=\curl.$ Note that then $D_{0}=-\curl_{0}$ and $D=-\curl.$
\end{example}
As the previous example illustrates, we want to interpret $G_{0}$
and $D_{0}$ as abstract differential operators with vanishing boundary
conditions, while $G$ and $D$ are the respective differential operators
without any boundary condition. This motivates the following definition.
\begin{defn*}
We define the spaces%
\footnote{For a closed linear operator $C$ we denote by $\mathcal{D}_{C}$
its domain, equipped with the graph-norm of $C$.%
} 
\[
\mathcal{BD}(G)\coloneqq\left(\mathcal{D}(G_{0})\right)^{\bot_{\mathcal{D}_{G}}},\quad\mathcal{BD}(D)\coloneqq\left(\mathcal{D}(D_{0})\right)^{\bot_{\mathcal{D}_{D}}},
\]
where the orthogonal complements are taken in $\mathcal{D}_{G}$ and
$\mathcal{D}_{D}$, respectively. We call $\mathcal{BD}(G)$ and $\mathcal{BD}(D)$
\emph{abstract boundary data spaces associated with $G$ }and\emph{
$D$, }respectively. Consequently, we can decompose $\mathcal{D}_{G}=\mathcal{D}_{G_{0}}\oplus\mathcal{BD}(G)$
and $\mathcal{D}_{D}=\mathcal{D}_{D_{0}}\oplus\mathcal{BD}(D).$ We
denote by $\pi_{\mathcal{BD}(G)}:\mathcal{D}_{G}\to\mathcal{BD}(G)$
and by $\pi_{\mathcal{BD}(D)}:\mathcal{D}_{D}\to\mathcal{BD}(D)$
the corresponding orthogonal projections. In consequence, $\pi_{\mathcal{BD}(G)}^{\ast}:\mathcal{BD}(G)\to\mathcal{D}_{G}$
and $\pi_{\mathcal{BD}(D)}^{\ast}:\mathcal{BD}(D)\to\mathcal{D}_{D}$
are the canonical embeddings and we set $P_{\mathcal{BD}(G)}\coloneqq\pi_{\mathcal{BD}(G)}^{\ast}\pi_{\mathcal{BD}(G)}:\mathcal{D}_{G}\to\mathcal{D}_{G}$
as well as $P_{\mathcal{BD}(D)}\coloneqq\pi_{\mathcal{BD}(D)}^{\ast}\pi_{\mathcal{BD}(D)}:\mathcal{D}_{D}\to\mathcal{D}_{D}$.
An easy computation gives 
\[
\mathcal{BD}(G)=[\{0\}](1-DG),\quad\mathcal{BD}(D)=[\{0\}](1-GD)
\]
 and thus, $G[\mathcal{BD}(G)]\subseteq\mathcal{BD}(D)$ as well as
$D[\mathcal{BD}(D)]\subseteq\mathcal{BD}(G).$ We set 
\begin{align*}
\stackrel{\bullet}{G}:\mathcal{BD}(G) & \to\mathcal{BD}(D),x\mapsto Gx\\
\stackrel{\bullet}{D}:\mathcal{BD}(D) & \to\mathcal{BD}(G),x\mapsto Dx
\end{align*}
and observe that both are unitary operators satisfying $\left(\stackrel{\bullet}{G}\right)^{\ast}=\stackrel{\bullet}{D}$
(see \cite[Section 5.2]{Picard2012_comprehensive_control} for details).
\end{defn*}
Having these notions at hand, we are ready to state the main result
of \cite{Trostorff2013_bd_maxmon}.
\begin{thm}[{\cite[Theorem 3.1]{Trostorff2013_bd_maxmon}}]
\label{thm:char_max_mon} Let $G_{0},D_{0},G$ and $D$ be as above
and let 
\[
A\subseteq\left(\begin{array}{cc}
0 & D\\
G & 0
\end{array}\right):\mathcal{D}(A)\subseteq H_{0}\oplus H_{1}\to H_{0}\oplus H_{1}
\]
be a (possibly nonlinear) restriction of $\left(\begin{array}{cc}
0 & D\\
G & 0
\end{array}\right):\mathcal{D}(G)\times\mathcal{D}(D)\subseteq H_{0}\oplus H_{1}\to H_{0}\oplus H_{1},(u,w)\mapsto(Dw,Gu).$ Then $A$ is maximal monotone, if and only if there exists a maximal
monotone relation $h\subseteq\mathcal{BD}(G)\oplus\mathcal{BD}(G)$
such that 
\[
\mathcal{D}(A)=\left\{ (u,w)\in\mathcal{D}(G)\times\mathcal{D}(D)\,\left|\,\left(\pi_{\mathcal{BD}(G)}u,\stackrel{\bullet}{D}\pi_{\mathcal{BD}(D)}w\right)\in h\right.\right\} .
\]
We call $h$ the \emph{boundary relation associated with $A$.}
\end{thm}

\subsection{A class of block operator matrices in system theory}

In \cite{Staffans2012_phys} the following class of block operator
matrices is considered: Let $E,E_{0},H,U$ be Hilbert spaces such
that $E_{0}\subseteq E$ with dense and continuous embedding. Moreover,
let $L\in L(E_{0},H)$ and $K\in L(E_{0},U)$ such that 
\[
\left(\begin{array}{c}
L\\
K
\end{array}\right):E_{0}\subseteq E\to H\oplus U
\]
is closed. This assumption particularly yields that the norm on $E_{0}$
is equivalent to the graph norm of $\left(\begin{array}{c}
L\\
K
\end{array}\right).$ We define $L^{\diamond}\in L(H,E_{0}')$ and $K^{\diamond}\in L(U,E_{0}')$
by $\left(L^{\diamond}x\right)(w)\coloneqq\langle x|Lw\rangle_{H}$
and $\left(K^{\diamond}u\right)(w)\coloneqq\langle u|Kw\rangle_{U}$
for $x\in H,w\in E_{0},u\in U$ and consider the following operator
\begin{equation}
A\subseteq\left(\begin{array}{cc}
K^{\diamond}K & -L^{\diamond}\\
L & 0
\end{array}\right):\mathcal{D}(A)\subseteq E\oplus H\to E\oplus H\label{eq:A_Staffans}
\end{equation}
with $\mathcal{D}(A)\coloneqq\left\{ (u,w)\in E_{0}\times H\,|\, K^{\diamond}Ku-L^{\diamond}w\in E\right\} ,$
where we consider $E\cong E'\subseteq E_{0}'$ as a subspace of $E_{0}'.$
We recall the following result from \cite{Staffans2012_phys}, which
we present in a slight different formulation%
\footnote{We note that in \cite{Staffans2012_phys} an additional operator $G\in L(E_{0},E_{0}')$
is incorporated in $A$, which we will omit for simplicity.%
}.
\begin{thm}[{\cite[Theorem 1.4]{Staffans2012_phys}}]
\label{thm:Staffans} The operator $A$ defined above is maximal
monotone. \end{thm}
\begin{rem}
We remark that in \cite[Theorem 1.4]{Staffans2012_phys} the operator
$-A$ is considered and it is proved that $-A$ is the generator of
a contraction semigroup. 
\end{rem}
We note that operators of the form (\ref{eq:A_Staffans}) were applied
to discuss boundary control problems. For instance in \cite{Tucsnak_Weiss2003,Tucsnak2003_thinair2}
the setting was used to study the wave equation with boundary control
on a smooth domain $\Omega\subseteq\mathbb{R}^{n}$. In this case
the operator $L$ was a suitable realization of the gradient on $L_{2}(\Omega)$
and $K$ was the Dirichlet trace operator. More recently, Maxwell's
equations on a smooth domain $\Omega\subseteq\mathbb{R}^{3}$ with
boundary control were studied within this setting (see \cite{Staffans2013_Maxwell}).
In this case $L$ was a suitable realization of $\curl$, while $K$
was the trace operator mapping elements in $\mathcal{D}(L)$ to their
tangential component on the boundary. \\
In both cases, there exist two closed operators $G_{0}:\mathcal{D}(G_{0})\subseteq E\to H,\: D_{0}:\mathcal{D}(D_{0})\subseteq H\to E$
with $G_{0}\subseteq-(D_{0})^{\ast}\eqqcolon G$ such that $G_{0}\subseteq L\subseteq G$
and $K|_{\mathcal{D}(G_{0})}=0$ (cp. \prettyref{exa:G_D}). It is
the purpose of this paper to show how the operators $A$ in \prettyref{eq:A_Staffans}
and in \prettyref{thm:char_max_mon} are related in this case.

\section{Main result}

Let $E,H$ be Hilbert spaces and $G_{0}:\mathcal{D}(G_{0})\subseteq E\to H$
and $D_{0}:\mathcal{D}(D_{0})\subseteq H\to E$ be densely defined
closed linear operators with $G_{0}\subseteq-(D_{0})^{\ast}\eqqcolon G$
and $D_{0}\subseteq-(G_{0})^{\ast}\eqqcolon D$. 

\begin{hyp}\label{hyp: standing} We say that two Hilbert spaces
$E_{0},U$ and two operators $L\in L(E_{0},H)$ and $K\in L(E_{0},U)$
satisfy the hypothesis, if 

\begin{enumerate}[(a)]

\item $E_{0}\subseteq E$ dense and continuous.

\item $\left(\begin{array}{c}
L\\
K
\end{array}\right):E_{0}\subseteq E\to H\oplus U$ is closed.

\item $G_{0}\subseteq L\subseteq G$ and $K|_{\mathcal{D}(G_{0})}=0$. 

\end{enumerate}

\end{hyp}
\begin{lem}
\label{lem:D_extends} Assume that $E_{0},U$ and $L,K$ satisfy the
hypothesis. Let $(u,w)\in E_{0}\times H$ such that $K^{\diamond}Ku-L^{\diamond}w\in E.$
Then $w\in\mathcal{D}(D)$ and $Dw=K^{\diamond}Ku-L^{\diamond}w.$\end{lem}
\begin{proof}
For $v\in\mathcal{D}(G_{0})$ we compute 
\begin{align*}
\langle w|G_{c}v\rangle_{H} & =\langle w|Lv\rangle_{H}\\
 & =\left(L^{\diamond}w\right)(v)\\
 & =(-K^{\diamond}Ku+L^{\diamond}w)(v)+(K^{\diamond}Ku)(v)\\
 & =\langle-K^{\diamond}Ku+L^{\diamond}w|v\rangle_{E}+\langle Ku|Kv\rangle_{U}\\
 & =\langle-K^{\diamond}Ku+L^{\diamond}w|v\rangle_{E},
\end{align*}
where we have used $G_{0}\subseteq L$ and $Kv=0$. The latter gives
$w\in\mathcal{D}(G_{0}^{\ast})=\mathcal{D}(D)$ and $Dw=-G_{0}^{\ast}w=K^{\diamond}Ku-L^{\diamond}w.$
\end{proof}
The latter lemma shows that if the hypothesis holds and $A$ is given
as in \prettyref{eq:A_Staffans}, then $A$ is a restriction of $\left(\begin{array}{cc}
0 & D\\
G & 0
\end{array}\right)$, which, by \prettyref{thm:Staffans}, is maximal monotone. However,
such restrictions are completely characterized by their associated
boundary relation (see \prettyref{thm:char_max_mon}). The question,
which now arises is: can we characterize those boundary relations,
allowing to represent $A$ as in \prettyref{eq:A_Staffans}? The answer
gives the following theorem.
\begin{thm}
\label{thm:main} Let $A\subseteq\left(\begin{array}{cc}
0 & D\\
G & 0
\end{array}\right)$. Then the following statements are equivalent.

\begin{enumerate}[(i)]

\item  There exist Hilbert spaces $E_{0},U$ and operators $L\in L(E_{0},H),\, K\in L(E_{0},U)$
satisfying the hypothesis, such that 
\[
\mathcal{D}(A)=\{(u,w)\in E_{0}\times H\,|\, K^{\diamond}Ku-L^{\diamond}w\in E\}.
\]

\item There exists $h\subseteq\mathcal{BD}(G)\oplus\mathcal{BD}(G)$
maximal monotone and selfadjoint, such that 
\[
\mathcal{D}(A)=\{(u,w)\in\mathcal{D}(G)\times\mathcal{D}(D)\,|\,(\pi_{\mathcal{BD}(G)}u,\stackrel{\bullet}{D}\pi_{\mathcal{BD}(D)}v)\in h\}.
\]

\end{enumerate}
\end{thm}
We begin to prove the implication (i)$\Rightarrow$(ii). 
\begin{lem}
\label{lem:h}Assume (i) in \prettyref{thm:main} and set 
\[
h\coloneqq\left\{ (x,y)\in\mathcal{BD}(G)\oplus\mathcal{BD}(G)\,|\,\pi_{\mathcal{BD}(G)}^{\ast}x\in E_{0},K^{\diamond}K\pi_{\mathcal{BD}(G)}^{\ast}x-L^{\diamond}\pi_{\mathcal{BD}(D)}^{\ast}\stackrel{\bullet}{G}y\in E\right\} .
\]
Then $(u,w)\in\mathcal{D}(A)$ if and only if $(u,w)\in\mathcal{D}(G)\times\mathcal{D}(D)$
with $(\pi_{\mathcal{BD}(G)}u,\stackrel{\bullet}{D}\pi_{\mathcal{BD}(D)}w)\in h$. \end{lem}
\begin{proof}
Let $(u,w)\in\mathcal{D}(A).$ Then we know by \prettyref{lem:D_extends},
that $(u,w)\in\mathcal{D}(G)\times\mathcal{D}(D).$ We decompose $u=u_{0}+P_{\mathcal{BD}(G)}u,$
where $u\in\mathcal{D}(G_{0})\subseteq E_{0}.$ Since $u,u_{0}\in E_{0}$
we get that $\pi_{\mathcal{BD}(G)}^{\ast}\pi_{\mathcal{BD}(G)}u=P_{\mathcal{BD}(G)}u\in E_{0}.$
In the same way we decompose $w=w_{0}+P_{\mathcal{BD}(D)}w,$ where
$w_{0}\in\mathcal{D}(D_{0}).$ Since 
\[
\left(L^{\diamond}w_{0}\right)(z)=\langle w_{0}|Lz\rangle_{H}=\langle w_{0}|Gz\rangle_{H}=\langle-D_{c}w_{0}|z\rangle_{E}
\]
for each $z\in E_{0},$ we obtain $L^{\diamond}w=-D_{0}w_{0}\in E$
and thus, 
\begin{align}
K^{\diamond}K\pi_{\mathcal{BD}(G)}^{\ast}\pi_{\mathcal{BD}(G)}u-L^{\diamond}\pi_{\mathcal{BD}(D)}^{\ast}\stackrel{\bullet}{G}\stackrel{\bullet}{D}\pi_{\mathcal{BD}(D)}w & =K^{\diamond}KP_{\mathcal{BD}(G)}u-L^{\diamond}P_{\mathcal{BD}(D)}w\nonumber \\
 & =K^{\diamond}K\left(u-u_{0}\right)-L^{\diamond}(w-w_{0})\nonumber \\
 & =K^{\diamond}Ku-L^{\diamond}w-D_{c}w_{0}\in E,\label{eq:decomp}
\end{align}
where we have used $\stackrel{\bullet}{G}\stackrel{\bullet}{D}=1,$
$Ku_{0}=0$ and $(u,w)\in\mathcal{D}(A).$ Thus, we have $(\pi_{\mathcal{BD}(G)}u,\stackrel{\bullet}{D}\pi_{\mathcal{BD}(D)}w)\in h.$
\\
Assume now, that $(u,w)\in\mathcal{D}(G)\times\mathcal{D}(D)$ with
$(\pi_{\mathcal{BD}(G)}u,\stackrel{\bullet}{D}\pi_{\mathcal{BD}(D)}w)\in h.$
Since $u_{0}\coloneqq u-P_{\mathcal{BD}(G)}u\in\mathcal{D}(G_{0})\subseteq E_{0}$
and by assumption $P_{\mathcal{BD}(G)}u=\pi_{\mathcal{BD}(G)}^{\ast}\pi_{\mathcal{BD}(G)}u\in E_{0}$,
we infer that $u\in E_{0}.$ Moreover, decomposing $w=w_{0}+P_{\mathcal{BD}(D)}w$
with $w\in\mathcal{D}(D_{0})$ and using $D_{0}w_{0}=-L^{\diamond}w_{0}$
we derive that 
\[
K^{\diamond}Ku-L^{\diamond}w=K^{\diamond}K\pi_{\mathcal{BD}(G)}^{\ast}\pi_{\mathcal{BD}(G)}u-L^{\diamond}\pi_{\mathcal{BD}(D)}^{\ast}\stackrel{\bullet}{G}\stackrel{\bullet}{D}\pi_{\mathcal{BD}(D)}w+D_{c}w_{0}\in E
\]
by \prettyref{eq:decomp} and $(\pi_{\mathcal{BD}(G)}u,\stackrel{\bullet}{D}\pi_{\mathcal{BD}(D)}w)\in h$.
Hence, $(u,w)\in\mathcal{D}(A)$. 
\end{proof}
Although we already know that $h$ in the previous Lemma is maximal
monotone by \prettyref{thm:Staffans} and \prettyref{thm:char_max_mon},
we will present a proof for this fact, which does not require these
Theorems. 
\begin{prop}
\label{prop:h_max_mon}Assume (i) in \prettyref{thm:main} holds and
let $h\subseteq\mathcal{BD}(G)\oplus\mathcal{BD}(G)$ be as in \prettyref{lem:h}.
Then $h$ is linear and maximal monotone.\end{prop}
\begin{proof}
The linearity of $h$ is clear due to the linearity of all operators
involved. Let now $(x,y)\in h.$ Then we compute 
\begin{align*}
\Re\langle x|y\rangle_{\mathcal{BD}(G)} & =\Re\langle\pi_{\mathcal{BD}(G)}^{\ast}x|\pi_{\mathcal{BD}(G)}^{\ast}y\rangle_{E}+\Re\langle G\pi_{\mathcal{BD}(G)}^{\ast}x|G\pi_{\mathcal{BD}(G)}^{\ast}y\rangle_{E}\\
 & =\Re\langle\pi_{\mathcal{BD}(G)}^{\ast}x|\pi_{\mathcal{BD}(G)}^{\ast}y\rangle_{E}+\Re\langle L\pi_{\mathcal{BD}(G)}^{\ast}x|G\pi_{\mathcal{BD}(G)}^{\ast}y\rangle_{E}\\
 & =\Re\langle\pi_{\mathcal{BD}(G)}^{\ast}x|\pi_{\mathcal{BD}(G)}^{\ast}y\rangle_{E}+\Re\left(L^{\diamond}G\pi_{\mathcal{BD}(G)}^{\ast}y\right)(\pi_{\mathcal{BD}(G)}^{\ast}x)\\
 & =\Re\left(L^{\diamond}\pi_{\mathcal{BD}(D)}^{\ast}\stackrel{\bullet}{G}y+\pi_{\mathcal{BD}(G)}^{\ast}y\right)(\pi_{\mathcal{BD}(G)}^{\ast}x)\\
 & =\Re\left(L^{\diamond}\pi_{\mathcal{BD}(D)}^{\ast}\stackrel{\bullet}{G}y+\pi_{\mathcal{BD}(G)}^{\ast}y-K^{\diamond}K\pi_{\mathcal{BD}(G)}^{\ast}x\right)(\pi_{\mathcal{BD}(G)}^{\ast}x)+\\
 & \quad+\langle K\pi_{\mathcal{BD}(G)}^{\ast}x|K\pi_{\mathcal{BD}(G)}^{\ast}x\rangle_{U}.
\end{align*}

Since $\pi_{\mathcal{BD}(G)}^{\ast}x\in E_{0}$ and $K^{\diamond}K\pi_{\mathcal{BD}(G)}^{\ast}x-L^{\diamond}\pi_{\mathcal{BD}(D)}^{\ast}\stackrel{\bullet}{G}y\in E$,
we get from \prettyref{lem:D_extends} $K^{\diamond}K\pi_{\mathcal{BD}(G)}^{\ast}x-L^{\diamond}\pi_{\mathcal{BD}(D)}^{\ast}\stackrel{\bullet}{G}y=D\pi_{\mathcal{BD}(D)}^{\ast}\stackrel{\bullet}{G}y=\pi_{\mathcal{BD}(G)}^{\ast}y,$
since $\stackrel{\bullet}{D}\stackrel{\bullet}{G}=1.$ Thus, we obtain
\begin{align*}
\Re\langle x|y\rangle_{\mathcal{BD}(G)} & =\Re\left(L^{\diamond}\pi_{\mathcal{BD}(D)}^{\ast}\stackrel{\bullet}{G}y+\pi_{\mathcal{BD}(G)}^{\ast}y-K^{\diamond}K\pi_{\mathcal{BD}(G)}^{\ast}x\right)(\pi_{\mathcal{BD}(G)}^{\ast}x)+\\
 & \quad+\langle K\pi_{\mathcal{BD}(G)}^{\ast}x|K\pi_{\mathcal{BD}(G)}^{\ast}x\rangle_{U}.\\
 & =\langle K\pi_{\mathcal{BD}(G)}^{\ast}x|K\pi_{\mathcal{BD}(G)}^{\ast}x\rangle_{U}\geq0.
\end{align*}
This proves the monotonicity of $h$. To show that $h$ is maximal
monotone, it suffices to prove $(1+h)[\mathcal{BD}(G)]=\mathcal{BD}(G)$
by \prettyref{thm:minty}. Let $f\in\mathcal{BD}(G)$ and consider
the linear functional 
\[
E_{0}\ni x\mapsto\langle\pi_{\mathcal{BD}(G)}^{\ast}f|x\rangle_{E}+\langle G\pi_{\mathcal{BD}(G)}^{\ast}f|Lx\rangle_{H}.
\]
This functional is continuous and thus there is $w\in E_{0}$ with%
\footnote{Recall that the norm on $E_{0}$ is equivalent to the graph norm of
$\left(\begin{array}{c}
L\\
K
\end{array}\right)$.%
} 
\begin{equation}
\forall x\in E_{0}:\langle w|x\rangle_{E}+\langle Lw|Lx\rangle_{H}+\langle Kw|Kx\rangle_{U}=\langle\pi_{\mathcal{BD}(G)}^{\ast}f|x\rangle_{E}+\langle G\pi_{\mathcal{BD}(G)}^{\ast}f|Lx\rangle_{H}.\label{eq:riesz}
\end{equation}
In particular, for $x\in\mathcal{D}(G_{0})\subseteq E_{0}$ we obtain
that 
\begin{align*}
\langle Gw|G_{c}x\rangle_{H} & =\langle Lw|Lx\rangle_{H}\\
 & =\langle\pi_{\mathcal{BD}(G)}^{\ast}f|x\rangle_{E}+\langle G\pi_{\mathcal{BD}(G)}^{\ast}f|Lx\rangle_{H}-\langle w|x\rangle_{E}\\
 & =\langle\pi_{\mathcal{BD}(G)}^{\ast}f|x\rangle_{E}+\langle G\pi_{\mathcal{BD}(G)}^{\ast}f|G_{c}x\rangle_{H}-\langle w|x\rangle_{E}\\
 & =\langle\pi_{\mathcal{BD}(G)}^{\ast}f|x\rangle_{E}-\langle DG\pi_{\mathcal{BD}(G)}^{\ast}f|x\rangle_{E}-\langle w|x\rangle_{E}\\
 & =-\langle w|x\rangle_{E},
\end{align*}
where we have used $Kx=0$ and $DG\pi_{\mathcal{BD}(G)}^{\ast}f=\pi_{\mathcal{BD}(G)}^{\ast}f.$
The latter gives $w\in\mathcal{D}(D)$ and $DGw=w$ or, in other words,
$P_{\mathcal{BD}(G)}w=w.$ Set $u\coloneqq\pi_{\mathcal{BD}(G)}w$
and $v=f-u$. It is left to show that $(u,v)\in h.$ First, note that
$\pi_{\mathcal{BD}(G)}^{\ast}u=P_{\mathcal{BD}(G)}w=w\in E_{0}.$
Moreover, we compute for $x\in E_{0}$ using \prettyref{eq:riesz}
\begin{align*}
\left(K^{\diamond}K\pi_{\mathcal{BD}(G)}^{\ast}u-L^{\diamond}\pi_{\mathcal{BD}(D)}^{\ast}\stackrel{\bullet}{G}v\right)(x) & =\langle K\pi_{\mathcal{BD}(G)}^{\ast}u|Kx\rangle_{U}-\langle\pi_{\mathcal{BD}(D)}^{\ast}\stackrel{\bullet}{G}v|Lx\rangle_{H}\\
 & =\langle Kw|Kx\rangle_{U}-\langle G\pi_{\mathcal{BD}(G)}^{\ast}f|Lx\rangle_{H}+\langle G\pi_{\mathcal{BD}(G)}^{\ast}u|Lx\rangle_{H}\\
 & =\langle Kw|Kx\rangle_{U}-\langle G\pi_{\mathcal{BD}(G)}^{\ast}f|Lx\rangle_{H}+\langle Lw|Lx\rangle_{H}\\
 & =\langle\pi_{\mathcal{BD}(G)}^{\ast}f|x\rangle_{E}-\langle w|x\rangle_{E},
\end{align*}
which gives $K^{\diamond}K\pi_{\mathcal{BD}(G)}^{\ast}u-L^{\diamond}\pi_{\mathcal{BD}(D)}^{\ast}\stackrel{\bullet}{G}v=\pi_{\mathcal{BD}(G)}^{\ast}f-w\in E.$
This completes the proof.
\end{proof}
The only thing, which is left to show is that $h$ is selfadjoint. 
\begin{prop}
Assume (i) in \prettyref{thm:main} holds and let $h\subseteq\mathcal{BD}(G)\oplus\mathcal{BD}(G)$
be as in \prettyref{lem:h}. Then $h$ is selfadjoint.\end{prop}
\begin{proof}
We note that $h^{\ast}$ is monotone, since $h$ is maximal monotone
by \prettyref{prop:h_max_mon} and \prettyref{rem:max_mon}. Thus,
due to the maximality of $h,$ it suffices to prove $h\subseteq h^{\ast}.$
For doing so, let $(u,v),(x,y)\in h.$ Then 
\begin{align*}
\langle y|u\rangle_{\mathcal{BD}(G)} & =\langle\pi_{\mathcal{BD}(G)}^{\ast}y|\pi_{\mathcal{BD}(G)}^{\ast}u\rangle_{E}+\langle G\pi_{\mathcal{BD}(G)}^{\ast}y|G\pi_{\mathcal{BD}(G)}^{\ast}u\rangle_{H}\\
 & =\langle\pi_{\mathcal{BD}(G)}^{\ast}y|\pi_{\mathcal{BD}(G)}^{\ast}u\rangle_{E}+\langle\pi_{\mathcal{BD}(D)}^{\ast}\stackrel{\bullet}{G}y|L\pi_{\mathcal{BD}(G)}^{\ast}u\rangle_{H}\\
 & =\langle\pi_{\mathcal{BD}(G)}^{\ast}y|\pi_{\mathcal{BD}(G)}^{\ast}u\rangle_{E}+\left(L^{\diamond}\pi_{\mathcal{BD}(D)}^{\ast}\stackrel{\bullet}{G}y-K^{\diamond}K\pi_{\mathcal{BD}(G)}^{\ast}x\right)(\pi_{\mathcal{BD}(G)}^{\ast}u)+\\
 & \quad+\langle K\pi_{\mathcal{BD}(G)}^{\ast}x|K\pi_{\mathcal{BD}(G)}^{\ast}u\rangle_{U}.
\end{align*}
Using that $\pi_{\mathcal{BD}(G)}^{\ast}x\in E_{0}$ and $L^{\diamond}\pi_{\mathcal{BD}(D)}^{\ast}\stackrel{\bullet}{G}y-K^{\diamond}K\pi_{\mathcal{BD}(G)}^{\ast}x\in E$
we have $L^{\diamond}\pi_{\mathcal{BD}(D)}^{\ast}\stackrel{\bullet}{G}y-K^{\diamond}K\pi_{\mathcal{BD}(G)}^{\ast}x=-D\pi_{\mathcal{BD}(D)}^{\ast}\stackrel{\bullet}{G}y=-\pi_{\mathcal{BD}(D)}^{\ast}y$
according to \prettyref{lem:D_extends}. Thus, 
\[
\langle y|u\rangle_{\mathcal{BD}(G)}=\langle K\pi_{\mathcal{BD}(G)}^{\ast}x|K\pi_{\mathcal{BD}(G)}^{\ast}u\rangle_{U}.
\]
Repeating this argumentation and interchanging $y$ and $x$ as well
as $u$ and $v$, we get that 
\[
\langle v|x\rangle_{\mathcal{BD}(G)}=\langle K\pi_{\mathcal{BD}(G)}u|K\pi_{\mathcal{BD}(G)}^{\ast}x\rangle_{U}
\]
and hence, $\langle y|u\rangle_{\mathcal{BD}(G)}=\langle x|v\rangle_{\mathcal{BD}(G)}$,
which implies $h\subseteq h^{\ast}.$ 
\end{proof}
This completes the proof of (i)$\Rightarrow$(ii) in \prettyref{thm:main}.
To show the converse implication, we need the following well-known
result for selfadjoint relations, which for sake of completeness,
will be proved.
\begin{thm}[{\cite[Theorem 5.3]{Arens1961}}]
 Let $Y$ be a Hilbert space and $C\subseteq Y\oplus Y$ a selfadjoint
relation. Let $U\coloneqq\overline{[Y]C}$. Then there exists a selfadjoint
operator $S:[Y]C\subseteq U\to U$ such that 
\[
C=S\oplus\left(\{0\}\times U^{\bot}\right).
\]
\end{thm}
\begin{proof}
Due to the selfadjointness of $C,$ we have that
\begin{equation}
U=\overline{[Y]C}=\overline{[Y]C^{\ast}}=\left(C[\{0\}]\right)^{\bot}.\label{eq:U}
\end{equation}
We define the relation $S\coloneqq\left\{ (u,v)\in U\oplus U\,|\,(u,v)\in C\right\} $
and prove that $S$ is a mapping. First we note that $S$ is linear
as $C$ and $U$ are linear. Thus, it suffices to show that $(0,v)\in S$
for some $v\in U$ implies $v=0.$ Indeed, if $(0,v)\in S,$ we have
$(0,v)\in C$ and hence, $v\in C[\{0\}]=U^{\bot}$. Thus, $v\in U\cap U^{\bot}=\{0\}$
and hence, $S$ is a mapping. Next, we show that $C=S\oplus\left(\{0\}\times U^{\bot}\right).$
First note that $S\subseteq C$ as well as $\{0\}\times U^{\bot}\subseteq C$
by definition and hence, $S\oplus(\{0\}\times U^{\bot})\subseteq C$
due to the linearity of $C$. Let now $(u,v)\in C$ and decompose
$v=v_{0}+v_{1}$ for $v_{0}\in U,v_{1}\in U^{\bot}=C[\{0\}].$ Hence,
$(0,v_{1})\in C$ and $\left(u,v_{0}\right)=(u,v)-(0,v_{1})\in C.$
Moreover, $u\in U$ by \prettyref{eq:U} and thus, we derive $(u,v_{0})\in S$
and consequently, $(u,v)=(u,v_{0})+(0,v_{1})\in S\oplus\left(\{0\}\times U^{\bot}\right).$
Finally, we show that $S$ is selfadjoint. Using that $C=S\oplus\left(\{0\}\times U^{\bot}\right)$,
we obtain that 
\begin{align*}
(x,y)\in S^{\ast} & \Leftrightarrow x,y\in U\wedge\forall(u,v_{0})\in S:\langle v_{0}|x\rangle_{U}=\langle u|y\rangle_{U}\\
 & \Leftrightarrow x,y\in U\wedge\forall(u,v_{0})\in S,v_{1}\in U^{\bot}:\langle v_{0}+v_{1}|x\rangle_{Y}=\langle v_{0}|x\rangle_{Y}=\langle u|y\rangle_{Y}\\
 & \Leftrightarrow x,y\in U\wedge\forall(u,v)\in C:\langle v|x\rangle_{Y}=\langle u|y\rangle_{Y}\\
 & \Leftrightarrow x,y\in U\wedge(x,y)\in C^{\ast}=C\\
 & \Leftrightarrow(x,y)\in S,
\end{align*}
which gives $S=S^{\ast},$ i.e. $S$ is selfadjoint.\end{proof}
\begin{rem}
It is obvious that in case of a monotone selfadjoint relation $C$
in the latter theorem, the operator $S$ is monotone, too.\end{rem}
\begin{lem}
\label{lem:L_and_K}Assume (ii) in \prettyref{thm:main} and let $h=S\oplus\left(\{0\}\times U^{\bot}\right)$
with $U\coloneqq\overline{[\mathcal{BD}(G)]h}$ and $S:\left[\mathcal{BD}(G)\right]h\subseteq U\to U$
selfadjoint. We define the vectorspace 
\[
E_{0}\coloneqq\left\{ u\in\mathcal{D}(G)\,|\,\pi_{\mathcal{BD}(G)}u\in\mathcal{D}(\sqrt{S})\right\} \subseteq E
\]
and the operators $L:E_{0}\to H,w\mapsto Gw$ and $K:E_{0}\to U,u\mapsto\sqrt{S}\pi_{\mathcal{BD}(G)}u.$
Then the operator $\left(\begin{array}{c}
L\\
K
\end{array}\right):E_{0}\subseteq E\to H\oplus U$ is closed and $E_{0},U,L$ and $K$ satisfy the hypothesis, if we
equip $E_{0}$ with the graph norm of $\left(\begin{array}{c}
L\\
K
\end{array}\right)$.\end{lem}
\begin{proof}
Let $(w_{n})_{n\in\mathbb{N}}$ be a sequence in $E_{0}$ such that
$w_{n}\to w$ in $E$, $Gw_{n}=Lw_{n}\to v$ in $H$ and $\sqrt{S}\pi_{\mathcal{BD}(G)}w_{n}=Kw_{n}\to z$
in $U$ for some $w\in E,v\in H,z\in U.$ Due to the closedness of
$G$ we infer that $w\in\mathcal{D}(G)$ and $v=Gw.$ Thus, $(w_{n})_{n\in\mathbb{N}}$
converges to $w$ in $\mathcal{D}_{G}$ and hence, $\pi_{\mathcal{BD}(G)}w_{n}\to\pi_{\mathcal{BD}(G)}w$.
By the closedness of $\sqrt{S},$ we get $\pi_{\mathcal{BD}(G)}w\in\mathcal{D}(\sqrt{S})$
and $z=\sqrt{S}\pi_{\mathcal{BD}(G)}w$. Thus, $w\in E_{0}$ and $\left(\begin{array}{c}
L\\
K
\end{array}\right)w=\left(\begin{array}{c}
v\\
z
\end{array}\right)$ and hence, $\left(\begin{array}{c}
L\\
K
\end{array}\right)$ is closed. Thus, $E_{0}$ equipped with the graph norm of $\left(\begin{array}{c}
L\\
K
\end{array}\right)$ is a Hilbert space. Moreover, $E_{0},U,L$ and $K$ satisfy the hypothesis,
since clearly $\mathcal{D}(G_{0})\subseteq E_{0}\subseteq\mathcal{D}(G)$,
which gives $E_{0}\subseteq E$ dense and $G_{0}\subseteq L\subseteq G$.
Moreover, by definition we have $K|_{\mathcal{D}(G_{0})}=0$. 
\end{proof}
The only thing, which is left to show, is that $\mathcal{D}(A)$ is
given as in \prettyref{thm:main} (i). 
\begin{lem}
Assume that (ii) in \prettyref{thm:main} holds and let $E_{0},U$
and $K,L$ be as in \prettyref{lem:L_and_K}. Then 
\[
\mathcal{D}(A)=\left\{ (u,w)\in E_{0}\times H\,|\, K^{\diamond}Ku-L^{\diamond}w\in E\right\} .
\]
\end{lem}
\begin{proof}
Let $(u,w)\in\mathcal{D}(A),$ i.e. $(u,w)\in\mathcal{D}(G)\times\mathcal{D}(D)$
with $(\pi_{\mathcal{BD}(G)}u,\stackrel{\bullet}{D}\pi_{\mathcal{BD}(D)}w)\in h.$
Then, by definition of $U$ and $S$, we have that $\pi_{\mathcal{BD}(G)}u\in\mathcal{D}(S)\subseteq E_{0}$
and $\stackrel{\bullet}{D}\pi_{\mathcal{BD}(D)}w-S\pi_{\mathcal{BD}(G)}u\in U^{\bot}.$
Let $x\in E_{0}$ and set $w_{0}\coloneqq w-P_{\mathcal{BD}(D)}w\in\mathcal{D}(D_{0})$
as well as $x_{0}\coloneqq x-P_{\mathcal{BD}(G)}x\in\mathcal{D}(G_{0}).$
Then we compute
\begin{align*}
\left(K^{\diamond}Ku-L^{\diamond}w\right)(x) & =\langle Ku|Kx\rangle_{U}-\langle w|Lx\rangle_{H}\\
 & =\langle\sqrt{S}\pi_{\mathcal{BD}(G)}u|\sqrt{S}\pi_{\mathcal{BD}(G)}x\rangle_{U}-\langle w|Gx\rangle_{H}\\
 & =\langle S\pi_{\mathcal{BD}(G)}u-\stackrel{\bullet}{D}\pi_{\mathcal{BD}(D)}w|\pi_{\mathcal{BD}(G)}x\rangle_{\mathcal{BD}(G)}+\\
 & \quad+\langle\stackrel{\bullet}{D}\pi_{\mathcal{BD}(D)}w|\pi_{\mathcal{BD}(G)}x\rangle_{\mathcal{BD}(G)}-\langle w|Gx\rangle_{H}\\
 & =\langle P_{\mathcal{BD}(D)}w|GP_{\mathcal{BD}(G)}x\rangle_{H}+\langle DP_{\mathcal{BD}(G)}w|P_{\mathcal{BD}(G)}x\rangle_{E}-\langle w|Gx\rangle_{H}\\
 & =\langle P_{\mathcal{BD}(D)}w|GP_{\mathcal{BD}(G)}x\rangle_{H}+\langle DP_{\mathcal{BD}(G)}w|P_{\mathcal{BD}(G)}x\rangle_{E}-\\
 & \quad-\langle w|GP_{\mathcal{BD}(G)}x\rangle_{H}-\langle w|G_{0}x_{0}\rangle_{H}\\
 & =\langle-w_{0}|GP_{\mathcal{BD}(G)}x\rangle_{H}+\langle DP_{\mathcal{BD}(G)}w|P_{\mathcal{BD}(G)}x\rangle_{E}+\langle Dw|x_{0}\rangle_{E}\\
 & =\langle D_{0}w_{0}|P_{\mathcal{BD}(G)}x\rangle_{E}+\langle DP_{\mathcal{BD}(G)}w|P_{\mathcal{BD}(G)}x\rangle_{E}+\langle Dw|x_{0}\rangle_{E}\\
 & =\langle Dw|x\rangle_{E},
\end{align*}
where we have used $\pi_{\mathcal{BD}(G)}x\in\mathcal{D}(\sqrt{S})\subseteq U$
in the fourth equality. Thus, $K^{\diamond}Ku-L^{\diamond}w=Dw\in E.$
Moreover, $u\in E_{0}$ since $\pi_{\mathcal{BD}(G)}u\in\mathcal{D}(S)\subseteq\mathcal{D}(\sqrt{S})$.
This proves one inclusion. Let now $(u,w)\in E_{0}\times H$ with
$K^{\diamond}Ku-L^{\diamond}w\in E.$ Then by \prettyref{lem:D_extends}
$w\in\mathcal{D}(D)$ with $K^{\diamond}Ku-L^{\diamond}w=Dw$. We
need to prove that $\pi_{\mathcal{BD}(G)}u\in\mathcal{D}(S).$ We
already have $\pi_{\mathcal{BD}(G)}u\in\mathcal{D}(\sqrt{S})$, by
definition of $E_{0}$. Let now $v\in\mathcal{D}(\sqrt{S}).$ Then
we have 
\begin{align*}
\langle\sqrt{S}\pi_{\mathcal{BD}(G)}u|\sqrt{S}v\rangle_{U} & =\langle Ku|K\pi_{\mathcal{BD}(G)}^{\ast}v\rangle_{U}\\
 & =\left(K^{\diamond}Ku-L^{\diamond}w\right)(\pi_{\mathcal{BD}(G)}^{\ast}v)+\langle w|L\pi_{\mathcal{BD}(G)}^{\ast}v\rangle_{H}\\
 & =\langle Dw|\pi_{\mathcal{BD}(G)}^{\ast}v\rangle_{E}+\langle w|G\pi_{\mathcal{BD}(G)}^{\ast}v\rangle_{H}\\
 & =\langle w|G\pi_{\mathcal{BD}(G)}^{\ast}v\rangle_{\mathcal{D}_{D}}\\
 & =\langle\pi_{\mathcal{BD}(D)}w|\stackrel{\bullet}{G}v\rangle_{\mathcal{BD}(D)}\\
 & =\langle\stackrel{\bullet}{D}\pi_{\mathcal{BD}(D)}w|v\rangle_{U},
\end{align*}
which gives $\pi_{\mathcal{BD}(G)}u\in\mathcal{D}(S)$ with $S\pi_{\mathcal{BD}(G)}u=\stackrel{\bullet}{D}\pi_{\mathcal{BD}(D)}w.$
Hence, $(\pi_{\mathcal{BD}(G)}u,\stackrel{\bullet}{D}\pi_{\mathcal{BD}(D)}w)\in S\subseteq h,$
and thus, $(u,w)\in\mathcal{D}(A).$ 
\end{proof}

\section*{Acknowledgement}

The author would like to thank George Weiss, who asked the question
on the relation between the operators considered in \cite{Staffans2012_phys}
and \cite{Trostorff2013_bd_maxmon} during a workshop in Leiden.

\end{document}